\documentclass[12pt, twoside, leqno]{article}
\usepackage{amsmath,amsthm}
\usepackage{amssymb,latexsym}
\usepackage{enumerate}

\usepackage{amsfonts}

\usepackage[cp1250]{inputenc}

\numberwithin{equation}{section}

\newcommand{\su}{{\rm  supp\,\, }}

\newcommand{\p}{\mbox{\scriptsize $\rm{\Pi}$}}

\newcommand{\ddr}{${\cal D}'^{ \{ M_{p } \} }$}

\newcommand{\md}{{\cal D}^{\{M_{p }\}}}

\newcommand{\me}{{\cal E}^{\{M_{p}\}}}

\newtheorem{lemma}{Lemma}

\newtheorem{proposition}{Proposition}

\newtheorem{definition}{Definition}

\renewcommand{\(}{\left(}
\renewcommand{\)}{\right)}

\begin{document}

 \title {Sequential conditions of integrability of Roumieu ultradistributions}

 \author{Svetlana Mincheva-Kami\'nska}

%S. Mincheva-Kami\'nska \\

\maketitle

 {\hspace{-.6cm}\textsl{Institute of Mathematics,
 Faculty of Mathematics and Natural Sciences, \\
\hspace{.3cm} University of Rzesz\'ow, \\
\hspace{.3cm} Prof. Pigonia 1, 35-310 Rzesz\'ow, Poland} \\
 \hspace{.3cm} \textsl{email}: \emph{minczewa@ur.edu.pl}

\vspace{1.2cm}

 \begin{abstract}
  %We prove
  The main result of the paper is the equivalence of five
  %consider several
  general conditions for integrability
 of
 %two
Roumieu ultradistributions
%on $\mathbb{R}^d$
in the space $\mathcal{D}'^{\{M_p\}}$. The result is
% being
an analogue of the known theorems of P. Dierolf ans J. Voigt on integrable distributions
 and S. Pilipovi\'c on integrable ultradistributions of Beurling type.
 %and prove their equivalence.
 Two of our equivalent
 %discussed sequential
 conditions are based on the classes
 ${\mathbb U}^{\{M_p\}}$ and $\overline{\mathbb U}^{\{M_p\}}$\! of $\mathfrak{R}$-approximate units. We use the same classes
 %are used
 in our separate paper to introduce
 some
 sequential definitions of the convolution of Roumieu ultradistributions in $\mathcal{D}'^{\{M_p\}}$,
%which are
equivalent to the
%functional
definition of S. Pilipovi{\'c} and B. Prangoski of the convolution in $\mathcal{D}'^{\{M_p\}}$.
% given by S. Pilipovi{\'c} and B. Prangoski.
% which

%This allows us to prove there commutativity of
%the convolution of Roumieu ultradistributions and the ultradifferential operator $P(D)$
 %of class $\{M_p\}$.

 % analogous to the known definitions in the space $\mathcal{D}'$ of distributions
% and in the space $\mathcal{D}'^{(M_p)}$ of ultradistributions of Beurling type.



 \end{abstract}

 %%%%%%%%%%%%%%%%%%%%%%%%%%%%%%%%%%%%%%%%%%%%%%%%%%%%%%%%%%%%%%%%%%%%%%%%%
%%%%%%%%%%%%%%%%%%%%%%%%%%%%%%%%%%%%%%%%%%%%%%%%%%%%%%%%%%%%%%%%%%%%%%%%%%%%%%%%%%%%%%%%%%%

 \section{Introduction}\label{sec1}

Integrability of a function (generalized function) is strictly connected with the notion of the convolution of two
functions (generalized functions).
%In case of
For Schwartz distributions this interplay can be noticed already
%indicated
in pioneering works of C. Chevalley \cite{Che} and L. Schwartz \cite{Sch-Sem} and then in publications of their followers.
An important contribution concerning integrability of distributions was done by P. Dierolf and J. Voigt in \cite{DiVo}. They gave there a characterization
of integrable distributions that consists of several equivalent conditions. Among them there are two sequential conditions, expressed in terms of certain classes of approximate units.
The characterization appeared to be very efficient in investigations concerning the convolution of distributions; in particular, in the proof of the equivalence of
general definitions of the convolution of distributions considered by various authors (see \cite{Che}, \cite{Sch-Sem}, \cite{shi}, \cite{HOR}, \cite{VLA1}, \cite{Kam}, \cite{M-K}).
%consisting of several equivalent conditions in terms of sequential conditions

The sequential conditions that appeared in the mentioned characterization of integrable distributions referred
to the sequential approaches to the convolution of distributions (see \cite{AMS} and \cite{VLA2})
and inspired investigations in the functional and sequential theories of the convolution of ultradistributions (see \cite{CKP} and references given there).

An initiator of investigations concerning integrability and convolvability of ultradistributions was S. Pilipovi\'c.
He proved in \cite{Pil91} an analogue of the theorem of Dierolf and Voigt on integrability of ultradistributions of Beurling type
and his result was then applied to the convolution of ultradistributions of Beurling type (see \cite{Pil91} and \cite{CKP}).
% and references given there).

The aim of this paper is to prove a characterization of integrable ultradistributions of Roumieu  type (see Theorem in section 4) in which
two
%sequential
conditions of integrability appear, being counterparts of the sequential conditions of Dierolf and Voigt in \cite{DiVo}.
The conditions are given in terms of the classes of $\mathfrak{R}$-approximate units (Definition 4)
and special $\mathfrak{R}$-approximate units (Definition 5), where $\mathfrak{R}$ denotes the class of all numerical sequences  increasing to infinity,
important in H. Komatsu's investigations of ultradistributions in \cite{Kom3}.
% (see Lemmas 1 and 2 in section 2).

In \cite{SMK}, we apply Theorem proved in section 4 to show
% the equivalence of
that the
%functional
definition
of the convolution of Roumieu ultradistributions, given in \cite{PP}
and
%studied
investigated
%deeply
by P. Dimovski, S. Pilipovi{\'c}, B. Prangoski and J. Vindas
(see \cite{PP}, \cite{DPV}, \cite{DPrV}, \cite{DPPV}, \cite{PPV}), is equivalent to several sequential definitions
of Roumieu ultradistributions introduced in \cite{PP} by means of the mentioned classes of $\mathfrak{R}$-approximate units
and special $\mathfrak{R}$-approximate units.
% We apply the  equivalence shown in Theorem to prove an important property of the commutativity of
% the convolution of Roumieu ultradistributions and the ultradifferential operator $P(D)$
% of class $\{M_p\}$.

%%%%%%%%%%%%%%%%%%%%%%%%%%%%%%%%%%%%%%%%%%%%%%%%%%%%%%%%%%%%%%%%%%%%%%%%%%%%%%%%%%%%%%%%%%%

 \section{Prerequisities}

    In what follows we will consider
    %smooth
    (complex-valued) functions
    %complex-valued
    %$\mathcal{C}^{\infty}$-functions
    and Roumieu ultradistributions defined on $\mathbb{R}^d$ using the standard multi-dimensional notation
    in $\mathbb{R}^d$.
    %(or on an open subset of $\mathbb{R}^d$) using the standard multi-dimensional notation  in $\mathbb{R}^d$.

  To abbreviate the notation
  %mark the dimension of $\mathbb{R}^{d}$, which is essential in some situations,
  we will denote
 the considered spaces of test functions  and the corresponding spaces
 of Roumieu ultradistributions on $\mathbb{R}^d$, for a given index $d$, by adding $d$
 at the end of the respective symbol. Analogously, we will denote
 %Moreover, if necessary,
 the constant function $1$ on
 $\mathbb{R}^{d}$
 %will be denoted
 by $1_d$ and the value of the functional $T\in{\cal{D}'}^{\{M_p\}}_d$
 on $\varphi\in{\cal{D}}^{\{M_p\}}_d$ by $\langle T, \varphi\rangle_d$.

\smallskip
The considered spaces of test functions and Roumieu ultradistributions depend
%are defined by
on a given
sequence
$(M_{p})_{p\in\mathbb{N}_0}$ of positive numbers, satisfying certain conditions.
%Usually some of
In this note we will always assume the following three conditions:
%are imposed on the sequence $(M_{p})$:

\bigskip
 \noindent
 (M.1)\
 %(logarithmic convexity)
\quad \qquad $M^{2}_{p} \leq M_{p-1} M_{p+1},\qquad p \in \mathbb{N};$

 \bigskip
 \noindent
 (M.2)\
 %(stability under ultradifferential operator)
  \quad \qquad $M_{p}\leq AH^{p}M_qM_{p-q},\qquad p,\,q\in \mathbb{N}_0,\; q\leq p;$
% \medskip
 %\noindent
 %(M.2') \ (stability under differential operator)

 %\quad  \ \quad $M_{p}\leq AH^{p}M_{p-1},\qquad p\in \mathbb{N};$

 \bigskip
 \noindent
 (M.3)\
 %(strong non-quasi-analyticity)
\quad \qquad $\sum^{\infty}_{p=q+1} M_{p-1}M_{p}^{-1} \leq Aq M_{q}M_{q+1}^{-1},\quad q \in \mathbb{N}$

 %\medskip
 %\noindent
 %(M.3') \ (non-quasi-analyticity)

 %\quad  \ \qquad $\sum^{\infty}_{p=1} M_{p-1}M_{p}^{-1} < \infty,$
 \bigskip
 \noindent
 for certain constants $A>0$ and $H>0$. Clearly, we may and will assume that $H\geq1$.

% Clearly, conditions (M.2') and (M.3') are particular cases of conditions (M.2) and (M.3),
% respectively.

 %For simplicity, we will assume in the whole paper that the sequence $(M_{p})$ satisfies the three
 %conditions (M.1), (M.2) and (M.3), not discussing which of them can be weakened or omitted
 %in the formulations of presented theorems.

It follows by induction from (M.1) that
\[
\frac{M_i}{M_{i-1}}\leq\frac{M_{p+i}}{M_{p+i-1}}, \qquad p\in \mathbb{N}_0, \ i\in \mathbb{N},
\]
and thus
\[
\frac{M_q}{M_0} = \prod^q_{i=1}\frac{M_i}{M_{i-1}}\leq \prod^q_{i=1}\frac{M_{p+i}}{M_{p+i-1}} = \frac{M_{p+q}}{M_p},
      \qquad p\in \mathbb{N}_0, \ q\in \mathbb{N}.
\]
Consequently, (M.1) implies $ M_p\cdot M_q \leq M_0M_{p+q}$ for $p, q\in \mathbb{N}_0$.

For simplicity we assume in the sequel that $M_0 = 1$. With this assumption, the last inequality gets the form:
 \begin{equation}\label{Mpq}
   M_p\cdot M_q \leq M_{p+q}, \qquad p, q\in \mathbb{N}_0.
     \end{equation}

 %It will be convenient to
 We define the multi-dimensional version of the sequence
 $(M_{p})_{p\in\mathbb{N}_0}$ as follows:
 %to the multi-dimensional version
 %($M_k)_{k\in\mathbb{N}_0^d}$
 %by means of the formula:
% \vspace{-.2cm}
 \[
  M_k := M_{k_1+\ldots+k_d},\qquad k = (k_1,\ldots,k_d)\in \mathbb{N}_0^d.
% \vspace{-.2cm}
 \]
%Due to the extended notation we immediately get
This definition leads one to the following extended version of inequality (\ref{Mpq}):
 \begin{equation}\label{Mkjd}
   M_k\cdot M_l \leq M_{k+l}, \qquad k, l\in \mathbb{N}^d_0.
     \end{equation}

 Recall that the {\it associated function} of the sequence  $(M_p)$ is defined by
\[
M(\rho)= \sup_{p\in\mathbb{N}_0} \log_+\frac{\rho^p}{M_p},  \quad \rho>0.
\]
\smallskip

%For an arbitrary $k = (k_1,\ldots,k_d)\in\mathbb{N}_0^d$

We denote by $D^k$, for $k = (k_1,\ldots,k_d)\in\mathbb{N}_0^d$, the following differential operator:
%of the form

\[\displaystyle D^k=D_1^{k_1}\cdots D_d^{k_d}:=
\left(\frac{1}{i}\frac{\partial}{\partial x_1}\right)^{k_1}\cdots\left(\frac{1}{i}\frac{\partial}{\partial x_d}\right)^{k_d}.
\]
\smallskip

An important role in our considerations is played by the class of sequences $(r_{\!p})_{p\in\mathbb{N}_0}$
of real numbers, with $r_0=1$, which  mo\-no\-tono\-usly increase to infinity.
 We denote this class, after \cite{PP} and  \cite{DPPV}, by $\mathfrak{R}$.

%in which numerical sequences mo\-no\-tono\-usly increasing
%to infinity are involved. The class of numerical sequences mo\-no\-tono\-usly increasing
%to infinity
%$(r_{\!p})=
%$(r_{\!p})_{p\in\mathbb{N}_0}$ (with $r_0=1$)
%has been denoted by $\mathfrak{R}$ in \cite{PP} and  \cite{DPPV} and we preserve this notation in our paper.

 For every $(r_{\!p})\in\mathfrak{R}$ we call $(R_p)$ the \textit{product sequence}
 corresponding to  $(r_{\!p})$ if its elements are of the form
 $R_p:=\prod_{i=0}^p r_i$ for $p\in\mathbb{N}_0$ (i.e. $R_0=1$).

 %Below we recall
 %In what follows,
 Later on we will use a result of H. Komatsu proved in \cite{Kom3}
(Lemma\,3.4 and Proposition\,3.5) that we formulate below in an equivalent
 shape of two entirely symmetric assertions (using the above notation for the product sequence
 corresponding to a given $(r_{\!p})\in\mathfrak{R}$):

\begin{lemma}\label{Kom1}
Let $a_p\geq 0$ for $p\in\mathbb{N}_0$.
The following two conditions are equivalent:

\bigskip
$(A_1)$\quad  $\sup_{p\in\mathbb{N}_0}\frac{a_p}{h^p} <\infty$\quad  for a certain
%constant
$h> 0$;

\bigskip
$(B_1)$\quad $\sup_{p\in\mathbb{N}_0}\frac{a_p}{R_k} <\infty$\quad for every
%sequence
$(r_p)\in \mathfrak{R}$.
\end{lemma}

\begin{lemma}\label{Kom2}
Let $a_k\geq 0$ for $k\in\mathbb{N}_0$.
The following two conditions are equivalent:

\bigskip
$(A_2)$\quad  $\sup_{p\in\mathbb{N}_0}\frac{a_p}{h^p} <\infty$\qquad  for every
%constant
$h> 0$;

\bigskip
$(B_2)$\quad $\sup_{p\in\mathbb{N}_0}\frac{a_p}{R_p} <\infty$\qquad for a certain
%sequence
$(r_p)\in \mathfrak{R}$.
 \end{lemma}

%\begin{remark} \label{remK}
Clearly, the above lemmas can be
%easily
extended to their  $d$-dimensional versions, i.e.
 concerning sequences $(a_k)_{k\in\mathbb{N}_0^d}$ of nonnegative numbers.
%\end{remark}

\medskip

We will also need the following simple lemma:

\begin{lemma}\label{Rk-ineq}
For every $(r_p)\in \mathfrak{R}$,
the following inequality holds:
\begin{equation}\label{Rk-ine}
R_{|k|}\cdot R_{|l|} \leq R_{|k+l|},\qquad\quad k,l\in \mathbb{N}_0^d,
\end{equation}
where $(R_p)$ is the product sequence corresponding to $(r_p)$.
\end{lemma}

\begin{proof}
% On the other hand,
 %or every $(r_{\!p})\in\mathfrak{R}$ we have
The inequality
 \[
 %begin{equation}\label{Rk-ineq}
   R_p\cdot R_q =\prod_{i=0}^p r_i\prod_{i=0}^q r_i \leq \prod_{i=0}^{p+q} r_i = R_{p+q},\qquad\quad p, q\in\mathbb{N}_0
\]
%end{equation}
%for $p, q\in\mathbb{N}_0$,
follows from the monotonicity of $(r_{\!p})$. Inequality (\ref{Rk-ine}) is its direct consequence.
 \end{proof}

%\bigskip

%It will be convenient to
Let us define formally the operation of multiplication of sequences from the class $\mathfrak{R}$
by positive scalars.

\medskip
%\textbf{

\begin{definition}\label{mult}
If $(r_{\!p})\in\mathfrak{R}$ and $\lambda>0$, then  by $\lambda (r_{\!p})$
% use further on the following notation for $\lambda>0$
%, \ m\in\mathbb{N}$
%and $(r_{\!p})\in\mathfrak{R}$:
we mean the sequence  $(\overline{r}_{\!p})$, where
% \begin{equation}\label{lamrp}
% \lambda\cdot (r_{\!p}) = (\overline{r}_{\!p})
% \end{equation}
$\overline{r}_{0}=1$ and
 $\overline{r}_{\!p} = \lambda r_{\!p}$ for $p\in\mathbb{N}$.
\end{definition}

% and
% \begin{equation}\label{mrp}
% (r_{\!p})^{\lambda} = (\widetilde{r}_{\!p}), \; \mbox{where} \;  \; \widetilde{r}_{0}=1   \;  \; \mbox{and} \;  \;
% \widetilde{r}_{\!p} = r^{\lambda}_{\!p}   \;  \; \mbox{for} \;  \; p\in\mathbb{N}.
% \end{equation}
% \noindent
\bigskip

% \begin{
\noindent
\textbf{Remark 1.}
%\label{rem0rn}
Assume that $(r_{\!p})\in\mathfrak{R}$.
 %$(r_{\!p})^{\lambda}\in\mathfrak{R}$ for every $\lambda > 0$.
 % and
 It is clear that then $\lambda(r_{\!p})\in\mathfrak{R}$ for $\lambda \geq 1$ and, on the other hand,
%Moreover,
$\lambda(r_{\!p})\in\mathfrak{R}$
 for $0 < \lambda < 1$ in case
 %$(r_{\!p})\in\mathfrak{R}$ and
 $r_1 > \lambda^{-1}$.
%\end{remark}

\bigskip

 We also define in the class $\mathfrak{R}$
%will use
the following partial ordering:

\medskip
\begin{definition}
%\textbf{Definition 2}.
\label{ordR}
%For two sequences
Let $(r_{\!p}), (s_{\!p})\in \mathfrak{R}$.
%are two sequences,
We will write $(s_{\!p})\prec (r_{\!p})$ whenever $s_{\!p} \leq r_{\!p}$ for all $p\in\mathbb{N}$ and
%there is a $k\in\mathbb{N}$ such that
$\limsup_{p\to\infty} r_p/s_p = \infty$.
%or, equivalently,
%(i.e. $\limsup_{p\to\infty} s_p/r_p^{1/k} = \infty$).
\end{definition}

% $(r_{\!p})$ and $(s_{\!p})\in\mathfrak{R}$ we
 %$m\in\mathbb{N}$.
% $will use notation
  %\begin{equation}\label{ordR}
 %(r_{\!p})\prec(s_{\!p}), \quad \mbox{when} \quad \limsup_{p\to\infty}\frac{s_{\!p}}{r_{\!p}}=\infty.
  %\end{equation}

%%%%%%%%%%%%%%%%%%%%%%%%%%%%%%%%%%%%%%%%%%%%%%%%%%%%%%%%%%%%%%%%%%%%%%%%%%%%%%%%%%%%%%%%%%%

\section{Spaces of ultradifferentiable functions}

%%%%%%%%%%%%%%%%%%%%%%%%%%%%%%%%%%%%%%%%%%%%%%%%%%%%%%%%%%%%%%%%%%%%%%%%%%%%%%%%%%%%%%%%%%%%%%%%%%%%%

   First let us denote formally
   \[
   \|\varphi\|_{\infty} := \sup_{x\in\mathbb{R}^{d}} |\varphi (x)|\qquad\mbox{and}\qquad \|\varphi\|_K := \sup_{x\in K} |\varphi (x)|
   \]
  \smallskip
 for a given
   %complex-valued
   %smooth
function $\varphi$ on $\mathbb{R}^{d}$ and for a compact set $K$ in $\mathbb{R}^{d}$.
% for a given sequence

 Now let $(M_p)$ be a given sequence, satisfying conditions (M.1), (M.2) and (M.3), let $K$ be a regular compact set
 in $\mathbb{R}^{d}$
 and let $h > 0$. By
 ${\cal{E}}^{\{M_{p}\}}_{K,h,d} $
we denote the locally convex space
 %(l.c.s.)
 of all smooth functions (i.e. $\mathcal{C}^{\infty}$-functions)
 $\varphi$ on $\mathbb{R}^{d}$ such that
 \begin{equation}\label{1qn0}
  %q_{K, h}(\varphi) :=
  \sup_{k \in \mathbb{N}^{d}_0}
 \frac{
 \| D^k\varphi \|_K} {h^{|k|}M_{k}} < \infty,
\end{equation}
 with the topology determined by
 the semi-norm $q_{K, h}$ defined by
 \begin{equation}\label{1qn}
  q_{K, h}(\varphi) :=
  \sup_{k \in \mathbb{N}^{d}_0}
 \frac{
 \| D^k\varphi \|_K} {h^{|k|}M_{k}}.
\end{equation}
 On the other hand,
  %the symbol
 ${\cal{D}}^{\{M_{p}\}}_{K,h,d}$
 means the Banach space of all smooth
 %$\cal{C}^{\infty}$-
 functions $\varphi$ satisfying (\ref{1qn0})
  and having  supports contained in  $K$,
  with the topology of the norm $q_{K, h}$ given by (\ref{1qn}).

 For a given sequence
 $(M_p)$, satisfying (M.1), (M.2), (M.3), and a regular compact set we define the
 locally convex space $ {\cal D}_{K,d}^{\{M_p\}}$
 %,  ${\cal D}_d^{\{M_p\}}$ and $ {\cal E}_d^{\{M_p\}}$
 of ultradifferentiable functions on $\mathbb{R}^{d}$ by the formula
 %as follows:
 %in the following way:
 \begin{equation}\label{D-K}
 {\cal D}_{K,d}^{\{M_p\}}
 \!:=\;
  \lim_{\begin{subarray}{c}
    \displaystyle\longrightarrow \\
    {\scriptsize h\rightarrow \infty}
    \end{subarray}} \,
 {\cal D}_{K,h,d}^{\{M_p\}}\, ,
 \end{equation}
and then we define the space ${\cal D}_d^{\{M_p\}}$ by
  \[
  %begin{equation}\label{Dd}
  {\cal D}_d^{\{M_p\}}
 \,\!:=\!
   \lim_{\begin{subarray}{c}
   \displaystyle\longrightarrow \\
   {\scriptsize K\subset\subset \mathbb{R}^d}
   \end{subarray}} \,
   {\cal D}_{K,d}^{\{M_p\}}.
 \]
 %end{equation}
Next the space ${\cal E}_d^{\{M_p\}}$ is defined by
 \[
 %begin{equation}\label{E}
 {\cal E}_d^{\{M_p\}} :=
 \lim_{\begin{subarray}{c}
   \displaystyle\longleftarrow \\
   {\scriptsize K\subset\subset \mathbb{R}^d}
   \end{subarray}} \,
  \lim_{\begin{subarray}{c}
    \displaystyle\longrightarrow \\
    {\scriptsize h\rightarrow \infty}
    \end{subarray}} \;
  {\cal E}_{K,h,d}^{\{M_p\}}\, ,
  \]
  %end{equation}
\noindent
where
%with
the symbol $K\subset\subset \mathbb{R}^d$
means that compact sets $K$
grow up to $\mathbb{R}^d.$
Moreover let
%, for a given $(M_p)$,
%we define
\[
 {\cal D}_{L^{\infty},d}^{\{M_p\}} \,:=\,
 \lim_{\begin{subarray}{c}
    \displaystyle\longrightarrow \\
    {\scriptsize h\rightarrow \infty}
    \end{subarray}} \,
 {\cal D}_{L^{\infty},h,d}^{\{M_p\}}\, ,
 \vspace{-.2cm}
 \]
 where ${\cal{D}}^{\{M_p\}}_{L^{\infty},h,d}$
 is the Banach space of
 all
 %$\mathcal{C}^{\infty}$-
 smooth functions $\varphi$
 on $\mathbb{R}^d$ such that
   \[
   %begin{equation}\label{DL-inf}
   %\|\varphi\|_{\infty,h} :=\,
   \sup\left\{\frac{\|D^k\varphi\|_{\infty}}{h^{k}M_k}\!:\ k \in \mathbb{N}^{d}_{0}\right\} < \infty
  \]
  %end{equation}
 with the norm $\|\cdot\|_{\infty,h}$ defined by
 \[
 %begin{equation}\label{DL-inf}
   \|\varphi\|_{\infty,h} :=\,
   \sup\left\{\frac{\|D^k\varphi\|_{\infty}}{h^{k}M_k}\!:\ k \in \mathbb{N}^{d}_{0}\right\}.
  \]
  %end{equation}

\smallskip
 For a given regular compact set
  $K\subset\mathbb{R}^d$, a sequence $(M_p)$ and  a sequence $(r_p)\in\mathfrak{R}$, we denote by
 ${\cal D}^{\{M_p\}}_{K,(r_p),d}$
 the Banach space of all
 %$\mathcal{C}^{\infty}$-
 smooth functions $\varphi$ on $\mathbb{R}^d$
 %having
 with supports contained in $K$ such that
    \[
    %begin{equation}\label{DK-ri}
    %\|\varphi\|_{K,(r_p)} :=
    \sup_{k \in \mathbb{N}^{d}_0}
        \frac{\| D^k\varphi\|_K}
        {R_{|k|}M_{k}} < \infty
\]
%end{equation}
 with the norm $\|\cdot\|_{K,(r_p)}$ defined by
     \begin{equation}\label{DK-ri}
    \|\varphi\|_{K,(r_p)} := \sup_{k \in \mathbb{N}^{d}_0}
        \frac{\| D^k\varphi\|_K}
        {R_{|k|}M_{k}}.
\end{equation}
 % above.

The following result is essentially due to Komatsu \cite{Kom3}, since it is a consequence of
his
%beautiful
Lemma \ref{Kom2} recalled above.

\begin{proposition} \label{p1}
    We have the equality
\begin{equation}\label{phi12-ine}
 \displaystyle {\cal D}_{K,d}^{\{M_p\}} =
 \lim_{\begin{subarray}{c}
   \displaystyle\longleftarrow \\
   {\scriptsize (r_p)\in\mathfrak{R}}
   \end{subarray}} \,
 \, {\cal D}_{K,(r_p),d}^{\{M_p\}}\, ,
\end{equation}
 where the space $\displaystyle {\cal D}_{K,d}^{\{M_p\}}$ is  defined in (\ref{D-K}).
\end{proposition}

 \begin{proof}
 The assertion is an immediate consequence of the definitions of $q_{K, h}$ in (\ref{1qn}) and of
 $\|\cdot\|_{K,(r_p)}$ in (\ref{DK-ri}) and, moreover,
 %, respectively, and
 %Part (I)
 of the $d$-dimensional version of Lemma \ref{Kom2}
 with
 \[
 a_k := \frac{\| D^k\varphi\|_K}{ M_{k}},\qquad  k\in \mathbb{N}^d_0
 \]
 for a given function $\varphi$ of the considered space.
   \end{proof}

 For given $(M_p)$ and
$(r_p)\in\mathfrak{R}$ we denote by
${\cal D}_{L^{\infty},(r_p),d}^{\{M_p\}}$
 the Banach space of all
 %$\mathcal{C}^{\infty}$-
 smooth functions
 $\varphi$ on $\mathbb{R}^d$
  such that
     \[
     %begin{equation}\label{D-ri}
        % \|\varphi\|_{(r_p)} :=
        \sup_{k \in \mathbb{N}^{d}_0}
       \frac{\|D^k\varphi\|_{\infty}}{{R_{|k|}M_{k}}} < \infty,
     \]
     %end{equation}
 with the norm $\|\cdot\|_{(r_p)}$ defined by
 %in (\ref{DK-ri}).
\begin{equation}\label{D-ri}
        \|\varphi\|_{(r_p)} := \sup_{k \in \mathbb{N}^{d}_0}
       \frac{\|D^k\varphi\|_{\infty}}{{R_{|k|}M_{k}}}.
     \end{equation}

\smallskip
% It follows from \cite{DPPV,PPV} that for a given sequence $(M_p)$
 %we define
We will need the following projective description of the space ${\cal{D}}^{\{M_p\}}_{L^{\infty},d}$
which follows from the results proved in \cite{DPPV,PPV}:
%is shown in \cite{DPPV,PPV}:

 For a given sequence $(M_p)$, satisfying  (M.1), (M.2), (M.3),
 the following equality holds
   \begin{equation}\label{Dtil}
   \displaystyle \displaystyle
  %\widetilde{{\cal D}}_{L^{\infty},d}^{\{M_p\}}
  {\cal{D}}^{\{M_p\}}_{L^{\infty},d}
  \,=\! \lim_{\begin{subarray}{c}
   \displaystyle\longleftarrow \\
   {\scriptsize (r_p)\in\mathfrak{R}}
   \end{subarray}} \,
 {\cal D}_{L^{\infty},(r_p),d}^{\{M_p\}}
  \end{equation}
%where the equality holds
in the sense of locally convex spaces.
%as l.c.s.

     We denote by $\dot{\cal{B}}^{\{M_p\}}_d$
 the completion of ${\cal{D}}^{\{M_p\}}_d$ in
	${\cal{D}}^{\{M_p\}}_{L^{\infty},d}$.
% and by\; $\dot{\!\widetilde{\cal{B}}}^{\{M_p\}}_d$
%    the completion of ${\cal{D}}^{\{M_p\}}_d$ in ${\widetilde{\cal{D}}}^{\{M_p\}}_{L^{\infty},d}$.

\bigskip

 %\begin{remark} \label{rem0rn}
\noindent
\textbf{Remark 2.}
 %Notice that, if necessary, w
 We may assume, if necessary,  that a given sequence $(r_{\!p})\in\mathfrak{R}$ satisfies, for a given constant $c>0$,
 the inequality $r_{\!p}>c$ for all $p\in\mathbb{N}$.

 In fact, since
 $r_{\!p}\nearrow\infty$ as $p\nearrow\infty$, there exists an index $p_0\in\mathbb{N}$ such that
 $r_{\!p}>c$ for $p>p_0$. Therefore we may replace the sequence $(r_{\!p})$ by  the sequence $( \widetilde{r}_{\!p})\in\mathfrak{R}$,
 defined by $\widetilde{r}_0=1$ and $\widetilde{r}_{\!p}:= r_{\!p+p_0}$ for all $p\in\mathbb{N}$, because
% Thus
$\|\varphi\|_{(r_{\!p})}<\infty$ if and only if $\|\varphi\|_{(\widetilde{r}_{\!p})} <\infty$ for all $\varphi\in{\cal D}_d^{\{M_p\}}$.
%  \end{remark}

\bigskip

%\section{Ultradifferentiable Functions}

\begin{proposition} \label{p2}
If $\varphi_1, \varphi_2\in {\cal{D}}^{\{M_p\}}_{L^{\infty},d}$,
then
%\cdot
$\varphi_1 \cdot \varphi_2\in{\cal{D}}^{\{M_p\}}_{L^{\infty},d}$.
Moreover, for every $(r_{\!p})\in\mathfrak{R}$ such that $r_1 > 2$
%$({r_{\!p}})/{2}\in\mathfrak{R}$,
the following inequality holds:
   \begin{equation}\label{phi-12-ine}
      \|\varphi_1\, \varphi_2\|_{(r_p)} \leq  \|\varphi_1\|_{{(r_{\!p})}\!/{2}}\cdot
      \|\varphi_2\|_{(r_{\!p})\!/{2}},
     \end{equation}
 where $({r_{\!p}})/{2}\in\mathfrak{R}$ is meant in the sense of Definition \ref{mult}.
 % (see Remark \ref{rem0rn}).
\end{proposition}

 \begin{proof}
 To prove the second part of the proposition assume that $(r_{\!p})$ is a sequence in $\mathfrak{R}$ such that $r_1 > 2$.  According to Definition \ref{mult} and Remark 1,
 %\ref{rem0rn}
 %the sequence
we have $(r_{\!p})\!/{2}\in\mathfrak{R}$,
 where the members $\overline{r}_p$ of the sequence $(r_{\!p})\!/{2}$ are defined by $\overline{r}_p := r_p/2$ for $p\in\mathbb{N}$ and $\overline{r}_0 := 1$.
 % according to Remark \ref{rem0rn}.
 %such that $(r_{\!p})\!/{2}=(\widetilde{r}_{\!p})\in\mathfrak{R}$ (see (\ref{lamrp}) for $\lambda=\frac{1}{2}$).

Clearly,
% $(\overline{r}_p)$
$\overline{R}_p = 2^{-p}R_p$ for $p\in\mathbb{N}_0$ and, consequently, $\overline{R}_{|k|} = 2^{-k}R_{|k|}$ for $k\in\mathbb{N}^d_0$, where
$(R_p)$ and $(\overline{R}_p)$ are the product sequences corresponding to
%the sequences
$(r_{\!p})$ and $(\overline{r}_p)$, respectively.

 Assume that $\varphi_1, \varphi_2\in {\cal{D}}^{\{M_p\}}_{L^{\infty},d}$.
Then  $\varphi_1, \varphi_2\in {\cal D}_{L^{\infty},(r_{\!p})\!/{2},d}^{\{M_p\}}$,
 according to (\ref{Dtil}).
 %the projective description of the spaces  ${\cal{D}}^{\{M_p\}}_{L^{\infty},d}$.
 % and ${\widetilde{\cal{D}}}^{\{M_p\}}_{L^{\infty},d}$ mentioned in Remark\,\ref{rem0}.
Now it follows from (\ref{D-ri}), applied for the sequence $(\overline{r}_p)$, that
\begin{equation}\label{Dlphi}
 \|D^l\varphi_i\|_{\infty}\leq  \overline{R}_{|l|}M_l\,\|\varphi_i\|_{(\overline{r}_p)}
 = 2^{-j}R_{|l|}M_l\,\|\varphi_i\|_{{(r_{\!p})}\!/{2}},
 \qquad   l\in\mathbb{N}^d_0,\, i=1, 2.
 \end{equation}
% for all $l\in\mathbb{N}^d_0$.
By Leibniz' formula, (\ref{Dlphi}), (\ref{Mkjd}) and (\ref{Rk-ine}), we get for arbitrary $k\in\mathbb{N}^d_0$ the estimates:
 \begin{eqnarray}
 \|D^k(\varphi_1\, \varphi_2)\|_{\infty}
 %\hspace{-.2cm}\hspace{-.2cm}\varphi_2)
\hspace{-.25cm} &\leq&\hspace{-.25cm}
 %\hspace{-.2cm}
 \sum_{0\leq j\leq k}
 %\hspace{-.2cm}
 \binom{k}{j}\|D^j\,\, \varphi_1\|_{\infty}\,
 %\cdot
 \|D^{k-j}\varphi_2\|_{\infty} \\
  &\leq& 2^{-k}\,\, \|\varphi_1\|_{{(r_{\!p})}\!/{2}}\,\,
  %\cdot
  \|\varphi_2 \|_{{(r_{\!p})}\!/{2}}\,
  %\cdot
       % \hspace{-.3cm}
       \sum_{0\leq j\leq k}
       %\hspace{-.2cm}
       \binom{k}{j} R_{|j|}M_j R_{|k-j|}M_{k-j}\nonumber\\
   &\leq& R_{|k|\,\,}M_{k}\,\, \|\varphi_1\|_{(r_{\!p})\!/2}\,\,
   %\cdot
    \|\varphi_2\|_{( r_{\!p})\!/2}. \label{fi12inf}\nonumber
\end{eqnarray}
%for each $k\in\mathbb{N}^d_0$.
%, due to (\ref{Mkjd}) and (\ref{Mkjd}).

% Therefore, it follows from (\ref{fi12inf}) that
Consequently,
 \[
 \frac{\|D^k(\varphi\phi)\|_{\infty}}{R_{|k|}M_k} \leq \|\varphi\|_{{(r_{\!p})}\!/{2}} \cdot \|\phi\|_{{(r_{\!p})}\!/{2}}<\infty,\qquad k\in\mathbb{N}^d_0,
 \]
%for all $k\in\mathbb{N}^d_0$,
which completes the proof of (\ref{phi-12-ine}) and the second part od the proposition.

The first part follows from the second one, according to Remark 2.
 \end{proof}

\bigskip

 %\begin{remark} \label{fiphiD}
 \noindent\textbf{Remark 3.}
 %Notice that the
 Assertion of Proposition\,\ref{p2} is true, in particular, for functions
 $\varphi_1, \varphi_2\in {\cal{D}}^{\{M_p\}}_d$
 and semi-norms of the form (\ref{DK-ri}).
 %\end{remark}

 \bigskip

  We will need
  %later
  the following simple consequence of Proposition \ref{p2}:

 \begin{lemma}\label{lem-tet}
 Let $\theta$ be a function from ${\cal{D}}_d^{\{M_p\}}$ such that $\theta(x)=1$ for all $x$ in some neighbourhood of $0$ in $\mathbb{R}^d, \ |x|\leq1$ and let
 $\theta_j(x)= \theta(\frac{x}{j})$ for $x\in\mathbb{R}^d$ and $j\in\mathbb{N}$. Then for each
 $(r_{\!p})\in\mathfrak{R}$ with $r_1\geq 2$ there exists a constant $C(\theta, (r_{\!p}))>0$ such that
 \begin{equation}\label{Cter}
 \|(1-\theta_l)\varphi\|_{(r_{\!p})} \leq C(\theta, (r_{\!p}))\|\varphi\|_{(r_{\!p})\!/{2}}, \qquad \varphi \in {\cal{D}}_d^{\{M_p\}}.
 \end{equation}
\end{lemma}

\begin{proof} It suffices to use Leibniz' formula and Proposition \ref{p2}.
  \end{proof}
%%%%%%%%%%%%%%%%%%%%%%%%%%%%%%%%%%%%%%%%%%%%%%%%%%%%%%%%%%%%%%%%%%%%%%%%%%%%%%%%%%%%%%%%%%%%%%%%%%%%%%%%%%%%%%%%

\section{Roumieu ultradistributions and their integrability}

%%%%%%%%%%%%%%%%%%%%%%%%%%%%%%%%%%%%%%%%%%%%%%%%%%%%%%%%%%%%%%%%%%%%%%%%%%%%%%%%%%%%%%%%%%%%%%%%%%%%%%%%%%%%%%%%%%%%%%%%%

% \begin{definition}
\noindent\textbf{Definition 3}.
%\label{ultraRo}
 % We call
 Elements of
  %denote
  the strong dual ${\cal{D'}}^{\{M_{p}\}}_d$ of the space ${\cal{D}}^{\{M_{p}\}}_d$
  %by \ ${\cal{D'}}^{\{M_{p}\}}_d$
  are called
  % it the {\it space
  {\it Roumieu ultradistributions} and elements of the strong dual ${\cal{D'}}^{\{M_p\}}_{L^{1},d}$
  of the space $\dot{{\cal{B}}}^{\{M_p\}}_d$
 %denoted by\ ${\cal{D'}}^{\{M_p\}}_{L^{1},d}$
 %is called the space of
 are called {\it Roumieu integrable ultradistributions}.
% \end{definition}

\bigskip

\noindent\textbf{Remark 4.}
%\label{rem1}
Since the space\ ${\cal{D}}^{\{M_p\}}$\ is
 dense in\ $\dot{{\cal{B}}}^{\{M_p\}}$\ and the respective inclusion mapping is continuous,
 %Consequently,
 the space\
 ${\cal{D'}}^{\{M_p\}}_{L^{1}}$\ of
 Roumieu integrable ultradistributions
        is a subspace of the space\ \ddr\ of Roumieu ultradistributions.
%\end{remark}
%In the sequel, we assume that the sequence $(M_p)$ satisfies conditions (M.1), (M.3) and (M.3), so
%we can use the projective description of the considered locally convex spaces of test functions.

\bigskip

%\begin{definition}
\noindent\textbf{Definition 4}.
%\label{RRAU}
 We define an {\it $\mathfrak{R}$-approximate unit} as a sequence
$(\p_{n})$ of ultradifferentiable functions
$\p_{n}\in\md_d$ which converges to $1$  in $\me_d$  and, for each
%sequence
$(r_{\!p})\in\mathfrak{R}$,
%we have
%such that the following property
%holds
%for every sequence $(r_{\!p})\in\mathfrak{R}$:
 \begin{equation}\label{RRap}
 \sup_{n \in \mathbb{N}} \|\p_n\|_{(r_{\!p})}
 %= \sup_{n \in \mathbb{N}} \sup_{k \in \mathbb{N}_0^{d}}
	 %(R_{|k|}M_{k})^{-1} \|D^k\p_n\|_{\infty}
 < \infty.
 \end{equation}
%for every sequence $(r_{\!p})\in\mathfrak{R}$ (see (\ref{D-ri})).
%where $(R_p)$ is the product sequence corresponding to $(r_p)$.
%\end{definition}

\bigskip

%\begin{definition}
%\noindent\textbf{
\noindent\textbf{Definition 5}.
%\label{RRsAU}
 If an $\mathfrak{R}$-approximate unit $(\p_{n})$ satisfies the additional condition that
 for each compact set $K\subset \mathbb{R}^{d}$ there is
 an index $n_0 \in\mathbb{N}$ such that $\p_{n}(x)=1$ for all $n\geq n_0$ and $x\in K$,
 then we call  $(\p_{n})$  a {\it special $\mathfrak{R}$-approximate unit}.
%\end{definition}

\bigskip

We denote by ${\mathbb U}^{\{M_p\}}_d$ the class of all $\mathfrak{R}$-approximate units on $\mathbb{R}^{d}$
and by $\overline{\mathbb U}^{\{M_p\}}_d$
%${\mathbb U}^{\{M_p\}}_d$ and
the class of all special $\mathfrak{R}$-approximate units on $\mathbb{R}^{d}$.
% by\; $\overline{\mathbb U}^{\{M_p\}}_d$.

 %\begin{remark}\label{ekvRAU}
 %As in the proof of Proposition\,\ref{p1}, one can show that the class ${\mathbb U}^{\{M_p\}}_d$
 %coincides with the class of all sequences
 %$(\p_n)$ of $\p_n\in\md_d$ converging to $1$ in $\me_d$ and such that
 %\begin{equation}\label{Rap}
 %\sup_{n \in \mathbb{N}} \sup_{k \in \mathbb{N}_0^{d}}
%	 (h^{k}M_{k})^{-1} \|D^k\p_n\|_{\infty} < \infty
% \end{equation}
% for some $h>0$.

 %Similarly, the class $\overline{\mathbb U}^{\{M_p\}}_d$ coincides with the class of all sequences
% $(\p_n)$ of $\p_n\in\md_d$ such that condition (\ref{Rap}) holds for some  $h>0$ and, in addition,
% for every compact set $K\subset \mathbb{R}^{d}$ there is such an index $n_0 \in\mathbb{N}$
%  that $\p_n(x)=1$ for all $n\geq n_0$ and $x\in K$.
 %\end{remark}

\bigskip

\noindent\textbf{Remark 5.}
%\label{DC}
 By the Denjoy-Carleman theorem, the spaces of ultradifferentiable functions defined in the previous section as well as the classes
 ${\mathbb U}_d^{\{M_p\}}$ and  $\overline{\mathbb U}_d^{\{M_p\}}$
 of approximate units contain sufficiently many members.
 %\end{remark}

\bigskip

%%%%%%%%%%%%%%%%%%%%%%%%%%%%%%%%%%%%%%%%%%%%%%%%%%%%%%%%%%%%%%%%%%%%%%%%%%%%%%%%%%%%%%%%%%%

% \section{Integrability of Roumieu Ultradistributions}

%%%%%%%%%%%%%%%%%%%%%%%%%%%%%%%%%%%%%%%%%%%%%%%%%%%%%%%%%%%%%%%%%%%%%%%%%%%%%%%%%%%%%%%%%%%

 Below we will
 %are going to
 prove our main theorem which is  a characterization,
consisting of
%several
five equivalent conditions,
  of integrable Roumieu ultradistributions.
 %in a way which is analogous to
 %which
 The theorem is an analogue of the theorem of Dierolf and  Voigt
 concerning integrable distributions (see \cite{DiVo}) and of the theorem of Pilipovi\'{c} concerning
 ultradistributions of Beurling type (see \cite{Pil91}).

\bigskip
%\begin{t
\noindent\textbf{Theorem.}
%\label{thI}
%Let
\textit{Assume that} $T\in{\cal{D}'}^{\{M_p\}}_d$. \textit{The following conditions
are equivalent:}

\smallskip
\indent
  $(a)$\;\; \textit{the functional} $T$ \textit{is continuous on} ${\cal{D}}^{\{M_p\}}_d$
  \textit{in the topology induced by}
$\dot{\cal{B}}^{\{M_p\}}_d$,
  \textit{i.e. there exist an}
  %sequence
  $(r_{\!p})\in\mathfrak{R}$ \textit{and a }
  %constant
  $C>0$ \textit{such that}
  %the inequality
\begin{equation}\label{a0}
|\langle T,\varphi\rangle| \leq C \|\varphi\|_{(r_p)}
\end{equation}
%holds
\textit{for all } $\varphi \in {\cal{D}}^{\{M_p\}}_d$;

\smallskip
$(b)$\;\; \textit{there exists an }
%sequence
$(r_{\!p})\in\mathfrak{R}$ \textit{such that
%with the property that
for every } $\varepsilon>0$
\textit{there is a regular compact set } $K\subset \mathbb{R}^d$
%such that the inequality
\textit{for which the inequalty}
\[
|\langle T,\varphi\rangle| \leq \varepsilon  \|\varphi\|_{(r_p)}
\]
\textit{holds true for all } $\varphi \in {\cal{D}}^{\{M_p\}}_d$ \textit{satisfying }
$\su \varphi \cap K = \emptyset$;

\smallskip
$(c)$\;\;
%for each $\(\p_n\) \in {\mathbb U} ^{\{M_p\}}_d$
    \textit{the sequence } $\left(\langle T,\p_n\rangle\right)$ \textit{is Cauchy for each } $\(\p_n\) \in {\mathbb U} ^{\{M_p\}}_d$;

\smallskip
$(d)$\;\;
%for every $\(\p_n\) \in \overline{\mathbb U}^{\{M_p\}}_d$
    \textit{the sequence } $\left(\langle T,\p_n\rangle\right)$ \textit{is Cauchy
    for each } $\(\p_n\) \in \overline{\mathbb U}^{\{M_p\}}_d$;

\smallskip
$(e)$\;\; \textit{there exist an }
% sequence
$(r_{\!p})\in\mathfrak{R}$, \textit{a }
%constant
$C>0$ \textit{and a regular compact }
$K \subset \mathbb{R}^d$ \textit{such that }
%inequality
(\ref{a0}) \textit{holds true
for all } $\varphi \in {\cal{D}}_d^{\{M_p\}}$ \textit{satisfying } $\su \varphi
\cap K=\emptyset$.
%\end{theorem}

\bigskip
%\begin{remark} \label{remtet}
 %Let $\theta \in {\cal{D}}_d^{\{M_p\}}$ be such that $\theta(x)=1$ for $x\in\mathbb{R}^d, \ |x|\leq1$, and put
 %$\theta_l(x)= \theta(\frac{x}{l})$ for $x\in\mathbb{R}^d$ and $l\in\mathbb{N}$. By Proposition \ref{p2},
 %if $(r_{\!p})\in\mathfrak{R}$ and $(r_{\!p})\!/{2}\in\mathfrak{R}$, then
 %there is a constant $C(\theta, (r_{\!p}))>0$ such that
 %\begin{equation}\label{Cter}
 %\|(1-\theta_l)\varphi\|_{(r_{\!p})} \leq C(\theta, (r_{\!p}))\|\varphi\|_{(r_{\!p})\!/{2}}, \qquad \varphi \in {\cal{D}}_d^{\{M_p\}}.
 %\end{equation}
%\end{remark}

 \begin{proof}

($a$)\,$\Rightarrow$($b$)\quad Assume that ($a$) is satisfied and fix a sequence $(r_{\!p})\in\mathfrak{R}$ and a constant
 $C>0$ such that (\ref{a0}) holds for all  $\varphi \in {\cal{D}}_d^{\{M_p\}}$.
 %At the same time,

% On the other hand, s
 Suppose that ($b$) is not true, i.e. there is an $\varepsilon_0>0$ such that  for every regular compact set
 $K \subset \mathbb{R}^d$ one can find a function $\varphi \in {\cal{D}}_d^{\{M_p\}}$ with $\su \varphi_{K} \cap K=\emptyset$
  such that
  \[
  \langle T, \varphi_{K}\rangle > \varepsilon_0 \|\varphi_{K}\|_{(r_p)}.
  \]
 Then there exists an increasing sequence of regular compact sets $K_n \subset \mathbb{R}^d$
 and a sequence of functions $\varphi_n\in {\cal{D}}_d^{\{M_p\}}$ such that
 $\su \varphi_n\subset \mbox{Int}K_{n+1}\setminus K_n, \; \|\varphi_n\|_{(r_p)}=1$ and
 $\langle T, \varphi_n\rangle > \varepsilon_0$ for all $n\in\mathbb{N}$. Put
 \[
 \psi_k := \sum_{n=1}^k \varphi_n, \quad k\in\mathbb{N}.
 \]
 Clearly, $\psi_k\in {\cal{D}}_d^{\{M_p\}}$ and,
 since the functions $\varphi_n$ have disjoined supports, we have
  $\|\psi_k\|_{(r_{\!p})}=1$ and $|\langle T, \psi_k\rangle| \leq C$ for all $k\in\mathbb{N}$. On the other hand,
 due to the assumption, we get
 \[
 \langle T, \psi_k\rangle > k\,{\varepsilon_0}  \qquad \mbox{for \ all} \; k\in\mathbb{N},
 \]
which contradicts the above estimate.

 ($b$)\,$\Rightarrow$($c$)\quad Fix  $\(\p_n\) \in {\mathbb U} ^{\{M_p\}}_d$ and assume that
 $(r_{\!p})\in\mathfrak{R}$ is a sequence with the property described in ($b$).

 By Remark 1,
 %\,\ref{rem0rn},
 we may additionally assume that $r_{\!p}>2$ for all $p\in\mathbb{N}$.
 Under the additional assumption, we have $(r_{\!p})\!/{2}\in\mathfrak{R}$
 and Lemma\,\ref{lem-tet} can be applied.

 Let $\theta$ and $\theta_l$ \ ($l\in\mathbb{N}$) be functions as described in Lemma\,\ref{lem-tet},
  so (\ref{Cter}) is true for some $C(\theta, (r_p))>0$. By (\ref{RRap}), we have
 \begin{equation}\label{N}
 N := 4\sup \{\|\p_n\|_{(r_{\!p})\!/{2}}: n\in\mathbb{N}\}<\infty.
 \end{equation}

 Fix $\varepsilon>0$ and, due to ($b$), choose a regular compact set $K \subset \mathbb{R}^d$
 such that
  \begin{equation}\label{Vfi}
 |\langle T,\varphi\rangle| \leq \frac{\varepsilon  \|\varphi\|_{(r_{\!p})}}{N C(\theta, (r_{\!p}))}\,
 \end{equation}
 for all $\varphi \in {\cal{D}}_d^{\{M_p\}}$ with $\su \varphi \cap K=\emptyset$.
 Now select $l\in\mathbb{N}$ so that $\theta_l(x)=1$ for all $x$ in a neighbourhood
 of $K$. Then the functions $\varphi_{l, k, n}:=(1-\theta_l)(\p_k-\p_n)$ have the property
  $\su \varphi_{l, k, n}\cap K=\emptyset$ for all $k, n\in\mathbb{N}$.
 Consequently, in view of (\ref{Vfi}), (\ref{Cter}) and (\ref{N}), we obtain
    \begin{eqnarray}
  |\langle T, \varphi_{l, k, n}\rangle| &\leq &
   \frac{\varepsilon}{N C(\theta, (r_p))} \|\varphi_{l, k, n}\|_{(r_{\!p})} \nonumber \\
  &\leq &  \frac{\varepsilon}{N}\|\p_k-\p_n\|_{(r_{\!p})\!/{2}}  \nonumber \\
  &\leq &  \frac{\varepsilon}{N}\left(\|\p_k\|_{(r_{\!p})\!/{2}}+\|\p_n\|_{(r_{\!p})\!/{2}}\right)\leq
  \frac{\varepsilon}{2}. \label{V1tee}
    \end{eqnarray}
 for all $k, l, n\in\mathbb{N}$. On the other hand,
 since $\(\p_n\)_{n\in\mathbb{N}}$ converges to $1$ in ${\cal{E}}_d^{\{M_{p}\}}$ and $\theta_l T\in{\cal{E}'}_d^{\{M_{p}\}}$,
 we have
  \begin{equation}\label{Vtee}
 |\langle \theta_lT,(\p_k-\p_n)\rangle| \leq  \frac{\varepsilon}{2}
 \end{equation}
 for sufficienly large $k$ and $n$.
 It remains to use (\ref{V1tee}) and (\ref{Vtee}) to get property ($c$).

  ($c$)\,$\Rightarrow$($d$)\quad is an obvious implication.

  ($d$)\,$\Rightarrow$($e$)\quad Assume $(d)$ and suppose that condition ($e$) does not hold.
  This means that for an arbitrary threesome: a sequence in $\mathfrak{R}$, a regular compact set in $\mathbb{R}^{d}$
  and a positive constant, one can find a function $\varphi \in{\cal{D}}_d^{\{M_p\}}$ with a support disjoint with the
  compact set for which (\ref{a0}) does not hold. In particular, for a given
  %an arbitrary
  $(r^1_{\!p})\in\mathfrak{R}$ and
  the compact ball  $K_1 := \overline{B}(0, 1)$ one can find a function $\varphi_1\in{\cal{D}}_d^{\{M_p\}}$
  with $\su \varphi_1 \cap K_1=\emptyset$ and $ |\langle T,\varphi_1\rangle| > \|\varphi_1\|_{(r^1_{\!p})}$.
 According to Definition\,1, choose a sequence $(r^2_{\!p})\in\mathfrak{R}$ such that $(r^2_{\!p})\prec(r^1_{\!p})$.
 %(see Definition \ref{ordR}).
 %Then f
 For $(r^2_{\!p})\in\mathfrak{R}$,
 the compact ball  $K_2 := \overline{B}(0, 2)$ and the constant $C_2:=4$ there exists a function $\varphi_2\in{\cal{D}}_d^{\{M_p\}}$
 such that $\su \varphi_2 \cap K_2=\emptyset$ and $ |\langle T,\varphi_2\rangle| > 2^2\|\varphi_2\|_{(r^2_{\!p})}$ and
 %And
 so on.
 % for every $m\in\mathbb{N}$ and $(r^{m-1}_{\!p})\in\mathfrak{R}$ we can take $(r^{m}_{\!p})\in\mathfrak{R}$ in a manner that
 %$(r^{m}_{\!p})\prec(r^{m-1}_{\!p})^{1/{m-1}}$. Thus for the sequence
 %
 %$(r^m_{\!p})^{1/m}\in\mathfrak{R}$, the compact ball  $K_m := \overline{B}(0, m)$ and the constant $C_m:=m^2$,
 In this way, we inductively construct a sequence of sequences $(r^m_{\!p})\in\mathfrak{R}$
 %such that
 with $(r^{m+1}_{\!p})\prec(r^{m}_{\!p})$
  and
  %there exists
  a sequence
  %$(\varphi_m)$
  of functions $\varphi_m\in{\cal{D}}_d^{\{M_p\}}$
  %such that
 with $\su \varphi_m \cap K_m=\emptyset$ such that
 %with the property
 %and
 \begin{equation}\label{Vm2}
  |\langle T,\varphi_m\rangle| > m^2 \|\varphi_m\|_{(r^m_{\!p})}>0
 \end{equation}
 for all $m\in\mathbb{N}$. Define
 %Put
 \begin{equation}\label{psim}
 \displaystyle \psi_m := \frac{\varphi_m}{m \|\varphi_m\|_{ (r^m_{\!p})}}\qquad\mbox{for}\;  m\in\mathbb{N}.
 \end{equation}
 Clearly,
 %we have
 $\psi_m\in{\cal{D}}_d^{\{M_p\}}$ and
 %, \; \;
 $\su \psi_m \cap K_m=\emptyset$. Moreover,
 \[
  \|\psi_m\|_{ (r^m_{\!p})}=\frac{1}{m} \;\;
 \mbox{and} \;\;  |\langle T, \psi_m\rangle| > m,
 \]
 by (\ref{Vm2}) and (\ref{psim}).

 %On the other hand, for every $(v_{\!p})\in\mathfrak{R}$
     %there is
% one can find an $m\in\mathbb{N}$ such that
 %$V_p=\prod_{i=0}^p v_{\!i} \geq \left(\prod_{i=0}^pr^m_{\!i}\right)^{1/m}$ and consequently,
 % for $p\in\mathbb{N}_0$. Consequently
% \[
% \|\psi_m\|_{ (v_{\!p})}\leq\|\psi_m\|_{ (r^m_{\!p})^{1/m}}\leq\frac{1}{m}.
% \]
% \vspace{.2cm}
 Now if $\(\p_n\)$
 %_{n\in\mathbb{N}}$
 is a given special $\mathfrak{R}$-approximate unit, then the sequence
 $\(\widetilde{\p}_n\)$, defined by
 %, of elements
 $\widetilde{\p}_n = {\p}_n + \psi_n$ for
 $n\in\mathbb{N}$, is also an $\mathfrak{R}$-approximate unit, because for every $(v_{\!p})\in\mathfrak{R}$
 we have $T_p=\prod_{i=0}^p v_{\!i} \geq \prod_{i=0}^pr^m_{\!i}$ for some $m\in\mathbb{N}$  and, consequently,
 % for $p\in\mathbb{N}_0$. Consequently
 \[
 \|\psi_m\|_{ (v_{\!p})}\leq\|\psi_m\|_{ (r^m_{\!p})}\leq\frac{1}{m}.
 \]
 %\vspace{.2cm}
 On the other hand, we have
 \[
 |\langle T, \widetilde{\p}_n\rangle - \langle T, \p_n\rangle| = |\langle T,\psi_n\rangle| > n,
 \]
 for $n\in\mathbb{N}$, which means that at least one of the sequences $\left(\langle T, \widetilde{\p}_n\rangle\right)$
 %_{n\in\mathbb{N}}$
 and $\left(\langle T, \p_n\rangle\right)$
 %_{n\in\mathbb{N}}$
 is not convergent. This contradicts the assumed condition $(D)$.
 % which is impossible.

 ($e$)\,$\Rightarrow$($a$)\quad Suppose that condition $(e)$ is satisfied for a sequence  $(r_{\!p})\in\mathfrak{R}$, a constant $C>0$ and
 a regular compact set $K \subset \mathbb{R}^d$. By Remark 2,
 %Using similar reasoning as in the proof of ($b$)\,$\Rightarrow$($c$)
 we may assume that  $(r_{\!p})\!/{2}\in\mathfrak{R}$.

 Select a relatively compact neighbourhood $U$ of $K$
 and a function $\theta\in{\cal{D}}_d^{\{M_p\}}$ such that $\su \theta\subset U$ and $\theta(x)=1$ for all $x$
 in a certain open  neighbourhood of $K$, contained in $U$. By ($e$), we have
 \begin{equation}\label{V1te}
 |\langle T, (1-\theta)\varphi\rangle| \leq  C \|(1-\theta)\varphi\|_{(r_{\!p})} \qquad \mbox{for} \;\; \varphi\in{\cal{D}}_d^{\{M_p\}}.
 \end{equation}

 The mapping
 \[
 {\cal{D}}_d^{\{M_p\}} \ni \varphi\mapsto \langle T, \theta\varphi\rangle\in\mathbb{C}
 \]
 defines the Roumieu ultadistribution $\theta T\in{\cal{E}'}_d^{\{M_{p}\}}$ given by
 \[
 \langle \theta T, \varphi\rangle := \langle T, \theta\varphi\rangle,\qquad \theta\in{\cal{D}}_d^{\{M_p\}}.
 \]
 Therefore there exists a constant $C_1>0$ such that
  \begin{equation}\label{VteK}
 |\langle \theta T, \varphi\rangle| \leq C_1\sup\left\{\frac{|D^k\varphi|}{R_{|k|}M_{k}}\!:\; x\in K, \ k\in\mathbb{N}^d_0\right\}
 \leq C_1 \|\varphi\|_{(r_{\!p})}
 \end{equation}
 for all $\varphi\in{\cal{D}}_d^{\{M_p\}}$. Hence
  \begin{eqnarray*}
 |\langle T, \varphi\rangle| &\leq&  |\langle T, (1-\theta)\varphi\rangle| +  |\langle \theta T, \varphi\rangle|
  \leq C \|(1-\theta)\varphi\|_{(r_{\!p})} +  C_1 \|\varphi\|_{(r_{\!p})}  \\
   &\leq&  C\cdot C(\theta, (r_{\!p})) \|\varphi\|_{(r_{\!p})\!/{2}} +  C_2 \|\varphi\|_{(r_{\!p})\!/{2}}
 \end{eqnarray*}
 for a certain $C_2>0$, in view of (\ref{V1te}), (\ref{VteK}) and (\ref{Cter}) in Lemma \ref{lem-tet}.

 The above estimate shows that $T$ is a continuous functional in ${\cal{D}}_d^{\{M_p\}}$
 in the topology induced from $\dot{\cal{B}}_d^{\{M_p\}}$.
  \end{proof}

\section*{Acknowledgements}

This work was partly supported by the Center for
Innovation and Transfer of Natural Sciences and Engineering
Knowledge of University of Rzesz\'ow.


\smallskip




\begin{thebibliography}{99}

 \bibitem{AMS} Antosik, P., Mikusi\'nski, J.,  Sikorski, R.:
              {\textit{Theory of Distributions. The Sequential Approach}},
              Elsevier-PWN,  Amsterdam-Warszawa, (1973).

\bibitem{CKP}
 Carmichael, R. D., Kami{\'n}ski, A., Pilipovi{\'c}, S.:
      {\textit{Boundary Values and Convolution in Ultradistribution Spaces}},
      World Scientific, New Jersey-London-Singapore-Beijing-Shanghai-Hong Kong-Taipei-Chennai
     (2007).

\bibitem{Che}
C.\ Chevalley, {\textit{Theory of Distributions}},
            {\em Lecture at Columbia University}, 1950-51.


 \bibitem{DiVo}  Dierolf, P.,  Voigt, J.:
             {\textit{Convolution and ${\mathcal S}'$-convolution of distributions}},
              Collect. Math. {\bf 29}, 185-196 (1978).

 \bibitem{DPV} Dimovski, P., Pilipovi{\'c}, S. and Vindas, J.:
  {\textit{New distribution spaces associated to translation-invariant Banach spaces}},
   Monatsh. Math., {\bf 177}, 495-515 (2015).

 \bibitem{DPPV} Dimovski, P., Pilipovi{\'c}, S., Prangoski, B. and Vindas, J.:
  {\textit{Convolution of ultradistributions and ultradistribution spaces associated to translation-invariant
  Banach spaces}}, Kyoto J. Math. {\bf 56}, nr 2, 401-440 (2016).

 %\bibitem{DPVe} Dimovski, P., Prangoski, B. and Velinov, D.:
  %{\textit{Multipliers and convolutions in the space of tempered ultradistributions}},
  %Novi Sad J. Math., {\bf{2}}, 1-18 (2014).

 \bibitem{DPrV} Dimovski, P., Prangoski, B. and Vindas, J.:
  {\textit{On a class of translation-invariant spaces of quasianalytic ultradistributions}},
  Novi Sad J. Math., {\bf 45(1)}, 143-175 (2015).

\bibitem{HOR} Horv\'ath, J.:
            {\textit{Topological Vector Spaces and Distributions}}, Vol. I,
            Addison-Wesley,  Reading-London (1966).

\bibitem{Kam} Kami\'nski, A.:
             {\textit{Convolution, product and Fourier transform of distributions}},
              Studia Math.  {74} (1982) 83-86.


% \bibitem{Kom1}
% Komatsu, H.:  {\textit{Ultradistributions, I: Structure theorems and a chracterization}}, \ J.\ Fac.\ Sci. \ Univ.\, Tokyo,\, Sect.\, I\ A\ Math.,\,
% {\bf 20}, 25-105, (1973).

% \bibitem{Kom2}
% Komatsu, H.:   {\textit{Ultradistributions, II:  The kernel theorem and ultradistribution with support submanifold}},
% J.\ Fac.\ Sci. \ Univ.\, Tokyo,\, Sect.\, I\ A\ Math.,\,  {\bf 24}, 607-628, (1977).


\bibitem{Kom3}
 Komatsu, H.:   {\textit{Ultradistributions, III: \ Vector \ valued \ ultradistributions and \ the \ theory \
of \ kernels}}, \ J.\ Fac.\ Sci. \ Univ.\, Tokyo,\, Sect.\, I\ A\ Math.,\,  {\bf 29}, 653-717, (1982).

\bibitem{M-K}
 Mincheva-Kaminska, S.:
  {\textit{Convolution of distributions in sequential approach}},
  Filomat, \textbf{28:8}, 1543--1557 (2014).

\bibitem{SMK}
 Mincheva-Kaminska, S.:
  {\textit{A sequential approach to the convolution of Roumieu ultradistributions}}, preprint.

%\bibitem{ORT-WAG}
% Ortner, N.,  Wagner, P.:
 %               {\textit{Distribution-Valued Analytic Functions - Theory and Applications}},
  %              Max-Planck-Institite,  LN 37,
   %             Leipzig,
    %            2008.

\bibitem{Pil91}
 Pilipovi{\'c}, S.:
 {\textit{On the convolution in the space of Beurling ultradistributions}},
 Comment. Math. Univ. St. Paul, {\bf 40}, 15-27 (1991).

\bibitem{PP}
 Pilipovi{\'c}, S., Prangoski, B.:
 {\textit{On the convolution of Roumieu ultradistributions through the $\varepsilon$ tensor product}},
 Monatsh. Math., {\bf 173}, 83-105 (2014).

%\bibitem{PP15}
 %Pilipovi{\'c}, S., Prangoski, B.:
 %{\textit{Anti-Wick and Weyl quantization on ultradistribution spaces}},
 %J. Math. Pure Appl, {\bf 103}, 472-503 (2015).

\bibitem{PPV}
 Pilipovi{\'c}, S., Prangoski, B., Vindas, J.:
 {\textit{On quasianalytic classes of Gelfand-Shilov type. Parametrix
  and convolution}},  J. Math. Pures Appl.  116: 174-210 (2018).


%\bibitem{PTT}
% Pilipovi{\'c}, S., Teofanov, N., Tomi{\'c}, F.:
% {\textit{On a class of ultradifferentiable functions }}
%  Novi Sad J. Math., {\bf 45(1)}, 125-142 (2015).

\bibitem{Sch-Sem}   Schwartz, L.:
               {\textit{D\'efinitions integrale de la convolution de deux distributions}},
              in: S\'eminaire Schwartz, Ann\'{e}e 1953--54.
              Produits tensoriels topologiques et d'espaces vec\-to\-riels to\-po\-logiques.
              Espaces vectoriels topologiques nucl\'eaires.
              Applications, Expose n$^{\circ}$ 22,
              Secr. math. Fac. Sci., Paris, 1954.

\bibitem{shi}Shiraishi, R.:
               {\textit{On the definition of convolution for distributions}},
               J. Sci. Hiroshima Univ. Ser. A 23, 19-32 (1959).

\bibitem{VLA1} Vladimirov, V. S.:
              {\textit{Equations of Mathematical Physics}},
              Nauka, Moscow 1967 (in Russian); English edition: Marcel Dekker,
              New York  1971.

\bibitem{VLA2} Vladimirov, V. S:
              {\textit{Methods of the Theory of Generalized Functions}},
              Taylor \& Francis, London-New York 2002.

% V. S. Vladimirov,
%Generalized Functions in
%Mathematical Physics, Mir Publishers 1979.

%\bibitem{VV}  Vu\v{c}kovi\'{c}, D., Vindas, J.:
 % {\textit{Eigenfunction expansions of ultradifferentiable functions and ultradistributions in $\mathbb{R}^n$}},
  %J. Pseudo-Differ. Oper. Appl., {\bf 7}, 519-531 (2016).


% \bibitem{Zio}  Zio{\l}o,  P.:
 %            {\textit{Geometric characterization of interpolatiion in the space of Fourier-Laplase
  %           transforms of ultradistributions of Roumieu type}},
   %           Collect. Math. {\bf 62}, 161-172 (2011).


\end{thebibliography}
\end{document}